\documentclass[11pt, reqno]{amsart}

\usepackage[a4paper,
            top=30mm,bottom=30mm,
            left=32mm,right=32mm]{geometry}

\usepackage[T1]{fontenc}

\usepackage{amsmath, amsthm, amsfonts, amssymb, mathtools}

\usepackage[shortlabels]{enumitem}
\usepackage[unicode=true]{hyperref}

\theoremstyle{plain}
\newtheorem{theorem}{Theorem}[section]
\newtheorem{lemma}[theorem]{Lemma}

\theoremstyle{definition}

\theoremstyle{remark}

\numberwithin{equation}{section}

\def\d{\mathrm{d}}
\def\eps{\varepsilon}
\def\N{\mathbb{N}}
\def\R{\mathbb{R}}
\def\I{\mathcal{I}}
\def\II{\I\times\I}
\def\L{\mathrm{L}}
\def\intI{\int_{\I}}
\def\iintII{\intI\intI}
\def\dd{\overline\nabla}
\def\Fcont{\mathcal{F}}
\def\Gcont{\mathcal{G}}
\def\Fkin{\Fcont_{\mathrm{kin}}}
\def\Gkin{\Gcont_{\mathrm{kin}}}
\def\Fmet{\Fcont_{\mathrm{met}}}
\def\Ltwozero{L^2_0}
\def\Lqzero{L^q_0}
\def\LinftyzeroN{L^\infty_{0,N}}
\def\Linftyzero{L^\infty_0}
\def\BB{\mathbb{B}}
\def\QQ{\mathcal{Q}}
\def\ZZ{\mathcal{Z}}
\def\half{\frac{1}{2}}
\def\Norm#1{\left\|#1\right\|}
\newcommand{\QN}{\QQ^N}
\newcommand{\ZN}{\ZZ^N}
\newcommand{\GraphG}{\mathbb{G}}
\newcommand{\Vset}{\mathbb{V}}
\newcommand{\Eset}{\mathbb{E}}

\title[]{A duality approach to the dense graph limit \\ for biological transportation networks}

\author[]{Nuno J. Alves}
\address[N. J. Alves]{CEMSE Division, King Abdullah University of Science and Technology, Thuwal 23955-6900, Saudi Arabia}
\email{nuno.januarioalves@kaust.edu.sa}

\author[]{Jan Haskovec}
\address[J. Haskovec]{CEMSE Division, King Abdullah University of Science and Technology, Thuwal 23955-6900, Saudi Arabia}
\email{jan.haskovec@kaust.edu.sa}

\begin{document}

\begin{abstract}
We develop a duality-based formulation of the dense graph limit for a variational model of biological transportation networks, where edge conductivities balance pumping power against metabolic cost.
In contrast to the pressure-based approach of our previous work, which required conductivities to be uniformly positive, the present formulation allows general nonnegative conductivity kernels.
The kinetic energy is defined through a dual variational principle, which remains meaningful for degenerate integrable kernels and assigns infinite energy when the associated nonlocal Poisson problem is not solvable.
Using this formulation, we prove $\Gamma$-convergence in the sense of Mosco of the semidiscrete network energies to a continuum energy on symmetric nonnegative kernels.
The convergence is obtained in the natural $L^\gamma$ topology dictated by the metabolic term. The $\Gamma$-$\liminf$ inequality follows directly from the dual formulation, while $\Gamma$-$\limsup$ recovery sequences are constructed by positive regularization of the conductivity kernels.
\end{abstract}


\maketitle
\thispagestyle{empty}

\section{Introduction}
In this paper we develop a duality-based formulation of the dense graph limit for the network formation model studied in~\cite{alves2026rigorous}.
This reformulation removes the a priori assumption of uniform positivity of the edge conductivities.
The discrete network formation model \cite{HuCaiPRL, HuCaiCMS} is posed on a sequence of undirected graphs
$\{\GraphG^N\}_{N\in\N}$,
\[
   \GraphG^N=(\Vset^N,\Eset^N),
\]
consisting of a set of vertices $\Vset^N = \{1,2,\dots,N\}$
and a set of unoriented edges (vessels) $\Eset^N \subseteq \Vset^N\times\Vset^N$.
We denote the corresponding adjacency matrices $W^N=(W_{ij}^N)_{i,j=1}^N$,
\[
   W_{ij}^N := \begin{cases}
     1 & \mbox{if } (i,j)\in \Eset^N,\\
     0 & \mbox{if } (i,j)\notin \Eset^N.
     \end{cases}
\]
For each edge $(i,j)\in \Eset^N$ we prescribe its length $\L_{ij}^N = \L_{ji}^N >0$.
At each vertex $i\in \Vset^N$ 
we prescribe a source/sink intensity $S_i^N\in\R$ of the material to be transported through the network.
We impose the global mass-balance condition
\[
   \sum_{i=1}^N S_i^N=0.
\]
We denote $C_{ij} = C_{ji} \geq 0$ the conductivity of the edge $(i,j)\in \Eset^N$,
and $P_i$ the pressure at vertex $i\in \Vset^N$.
We assume low Reynolds number of the flow through the network, so that the flow
$Q_{ij} = - Q_{ji}$ through the edge $(i,j)\in \Eset^N$ is proportional to its conductivity and the pressure drop between its endpoints,
\[
   Q_{ij}=C_{ij}\frac{P_j-P_i}{\L_{ij}^N}.
\]
Introducing the rescaled variable $B_{ij} := C_{ij} / \L_{ij}^N$,
local conservation of mass takes the form of the Kirchhoff law
\begin{equation}   \label{eq:K}
   - \frac{1}{N^2} \sum_{j=1}^N W_{ij}^N B_{ij} (P_j-P_i) =S_i^N
   \qquad \mbox{for all } i\in \Vset^N.
\end{equation}
The energy cost functional proposed in \cite{HuCaiPRL, HuCaiCMS} consists of a pumping power term and a metabolic cost term,
\begin{equation} \label{def:Fn}
F^N[B] := F^N_{\mathrm{kin}}[B] + \nu F^N_{\mathrm{met}}[B],
\end{equation}
with
\begin{equation}  \label{eq:Fnkin}
   F^N_{\mathrm{kin}}[B] := \frac{1}{2N^2}\sum_{i=1}^N\sum_{j=1}^N B_{ij}(P_j-P_i)^2 W_{ij}^N, 
\end{equation}
and
\begin{equation}  \label{eq:Fnmet} 
   F^N_{\mathrm{met}}[B] := \frac{1}{2\gamma N^2} \sum_{i=1}^N\sum_{j=1}^N B_{ij}^\gamma (\L_{ij}^N)^{\gamma+1} W_{ij}^N,
\end{equation}
where the pressures $P=(P_i)_{i\in\Vset^N}$ are a solution (unique up to an additive constant) of the linear system \eqref{eq:K}.
If \eqref{eq:K} is not solvable with the given conductivities $B = (B_{ij})_{(i,j)\in\Eset^N}$, we set $F^N[B] := \infty$.
Throughout the paper we assume that the metabolic exponent $\gamma\geq 1$,
and, for simplicity, we set the metabolic constant $\nu:=1$.
Let us note that for $\gamma\geq 1$ the problem \eqref{eq:K}-\eqref{def:Fn}
is convex (and strictly convex for $\gamma>1$); see, e.g., \cite[Proposition 1]{HV24}.

In \cite{alves2026rigorous} the rigorous dense graphon limit of the model \eqref{eq:K}-\eqref{def:Fn}
was derived as the size of the graph $N\to\infty$.
A graphon \cite{LovaszSzegedyJCTSB, Lovasz:2012} is a symmetric measurable function $w: [0,1]^2 \to [0,1]$.
In the context of our model we shall interpret the graphon $w$ as an appropriate limit of the sequence of
adjacency matrices $W^N$.
The approach developed in \cite{alves2026rigorous} relied on a reformulation of the energy functional \eqref{def:Fn} in terms of an integral functional acting on the space of piecewise constant functions on $[0,1]^2$. The limit was then established in the sense of $\Gamma$-convergence, with the Kirchhoff law \eqref{eq:K} represented in the limit by an elliptic integral equation.
Based on the $\Gamma$-limit, it was concluded that global minimizers of the discrete energy functional \eqref{eq:K}--\eqref{def:Fn} converge, in an appropriate sense, to global minimizers of the limiting (continuum) functional. For a general theory on $\Gamma$-convergence methods we refer to~\cite{Maso93}.

The results of \cite{alves2026rigorous} were established under the significantly restrictive assumption
that the conductivities $B_{ij}$ are uniformly positive, i.e.,
$B_{ij} \geq r >0$ for some $r > 0$ and all $(i,j)\in\Eset^N$, $N\in\N$. This assumption was adopted in order to ensure uniform ellipticity of the Kirchhoff law \eqref{eq:K}, which implied existence of a solution and uniform a priori bounds.

The main contribution of this paper is to remove this structural assumption by reformulating the kinetic energy in dual form.
This allows us to pose the variational problem on the cone of nonnegative conductivities, which is its natural setting.
We then construct the graphon $\Gamma$-limit for the resulting extended-value functional. On the uniformly positive class, the limiting problem can be identified with the minimization of the energy
\begin{equation} \label{def:Fcont}
\Fcont[b] := \Fkin[b] + \Fmet[b],
\end{equation}
with
\begin{equation}  \label{eq:Fcontkin}
   \Fkin[b] :=  \frac12 \iintII  b(x,y) \left(p(x)-p(y)\right)^2 \d w(x,y),
   \end{equation}
   and
   \begin{equation} \label{eq:Fcontmet}
   \Fmet[b] :=  \frac{1}{2\gamma} \iintII  b(x,y)^\gamma \ell(x,y)^{\gamma+1} \d w(x,y),
\end{equation}
where $\I := [0,1]$, $\ell\in L^\infty(\II)$ is the limit of the sequence of edge lengths, and $p\in L^2(\I)$ is a solution of the
nonlocal diffusion (Poisson-type) equation on $\I$,
\begin{equation}  \label{eq:Poisson}
   \intI b(x,y) (p(x)-p(y)) w(x,y) \,\d y = \sigma(x),
\end{equation}
where the datum $\sigma\in L^2(\I)$ represents the continuum intensity of sources and sinks.

Let us also point out that in \cite{alves2026rigorous} the limiting
conductivity $b=b(x,y)$ was required to belong to
$L^\omega(\II)$, with $\omega=\max\{2,\gamma\}$. In the present work
we obtain the variational convergence in the natural space $L^\gamma(\II)$, the exponent dictated by the metabolic energy.
In fact, the $\Gamma$-$\liminf$ condition only requires weak convergence in $L^1(\II)$.
The stronger assumption $b\in L^\gamma(\II)$ is needed for the construction of recovery sequences, in order to pass to the limit in the metabolic part of the energy.

Our derivation is based on the observation
that the kinetic part of the energy $\Fkin[b]$ admits a dual representation for any nonnegative conductivity kernel $b\in L^1(\II)$.
We take this dual representation as the definition of the kinetic term on the cone of symmetric nonnegative kernels.
The result is an extended-value functional which takes the value $\infty$ when equation \eqref{eq:Poisson} is not solvable
due to degeneracy of the conductivity kernel $b$.

Our main result establishes $\Gamma$-convergence in the sense of Mosco~\cite{mosco1969convergence,mosco1994composite,Maso93} of the semidiscrete energies to the continuum functional on the cone of nonnegative kernels, without any a priori uniform positivity assumption; see Theorem~\ref{thm:main}.

With the dual-based reformulation of the kinetic energy, the $\Gamma$-$\liminf$ inequality is proved
by passing to the limit directly in the affine pairing of the semidiscrete dual functional with bounded test functions.
The semidiscrete formulation serves as a convenient rewriting of the discrete problem.

The $\Gamma$-$\limsup$ inequality is obtained by the regularization of the conductivity kernel,
\[
   b^r:=b+r,
\]
with $r>0$. This is followed by a suitable choice of the recovery sequence as $r\to 0^+$, using the continuity of the limiting energy under this regularization.

Several recent papers have studied continuum limits of the model \eqref{eq:K}-\eqref{def:Fn}.
However, their constructions were embedded in the physical space rather than in the sense of graphons.
In \cite{HKM:CMS:2019} a continuum limit is derived formally
for the special case when the graph represents a two-dimensional rectangular grid.
This leads to an integral energy functional on the set of diagonal matrices
(permeability tensors), constrained by a Poisson equation.
In \cite{HKM:CommPDE:2019} this procedure is justified rigorously
in terms of the $\Gamma$-limit of the sequence of properly
rescaled and reformulated energy functionals.
Finally, in \cite{HMP:CMS:2022}, a formal
limit has been derived in the more general setting when
the discrete graphs are triangulations of a bounded two-dimensional domain.
This leads to a similar integral energy functional as in the previous case,
however, defined on the set of symmetric positive semidefinite tensors.
The functional is again constrained by a Poisson equation.

The paper is organized as follows. In Section~\ref{sec:notation} we give an overview of the notation.
In Section~\ref{sec:main} we formulate our main result.
In Section~\ref{sec:Poisson} we set up the function space for the Poisson equation \eqref{eq:Poisson} and introduce the dual
formulation of the kinetic energy functional.
In Section~\ref{sec:aux} we establish two auxiliary lemmas connecting the discrete and semidiscrete formulations of the problem.
Sections~\ref{sec:liminf} and \ref{sec:limsup} contain the proofs of the $\Gamma$-$\liminf$ and $\Gamma$-$\limsup$ inequalities.

\section{Notation}\label{sec:notation}
For any $N\in\N$ we shall denote $[N]:=\{1,\dots,N\}$.
We introduce the real intervals
\[
   \I := [0,1],
   \qquad
   \I_i^N := \left[\frac{i-1}{N},\frac{i}{N}\right] 
   \qquad\mbox{for }
   i\in [N]. 
\]
Moreover,
\[
   \R^N_0 := \left\{z\in \R^N:\, \sum_{i=1}^N z_i = 0\right\},
\]
and for $q\in[1,\infty]$,
\[
   \Lqzero(\I) := \left\{u\in L^q(\I):\, \intI u(x)\, \d x=0\right\}.
\]
We denote by $\LinftyzeroN(\I)$ the finite-dimensional subspace of $L^\infty_0(\I)$ consisting of functions that are constant on each interval $\I_i^N$, $i\in[N]$.

We introduce the set of symmetric matrices with nonnegative entries,
\begin{equation}  \label{def:BNplus}   
   \BB_+^N := \left\{B\in \R^{N\times N}:\, B_{ij} = B_{ji}\ge 0 \text{ for all } i,j\in [N] \right\},
\end{equation}
and for $r>0$ the uniformly positive class
\begin{equation}  \label{def:BNr}
   \BB_r^N := \left\{B\in \R^{N\times N}: \, B_{ij} = B_{ji} \ge r  \text{ for all } i,j\in [N] \right\}.
\end{equation}

For $q\in[1,\infty]$ we define the nonnegative cone
\begin{equation}   \label{def:Lqplus}
   L_+^q(\II):=\big\{b\in L^q(\II):\, b(x,y)=b(y,x) \text{ and } b(x,y)\ge 0\text{ a.e.\ on } \II\big\},
\end{equation}
and for $q\in[1,\infty]$ and $r>0$ the uniformly positive class
\begin{equation}   \label{def:Lqr}
   L_r^q(\II):=\big\{b\in L^q(\II):\, b(x,y)=b(y,x) \text{ and } b(x,y)\ge r\text{ a.e.\ on } \II\big\}.
\end{equation}

For any function $p:\I\to\R$ we denote
\begin{equation}\label{def:dd}
   \dd p(x,y):=p(y)-p(x), \qquad (x,y)\in\II.
\end{equation}

For $N\in\N$ we define the operators $\QN: \R^N \to L^\infty(\I)$ by
\[
   \QN[z](x):=\sum_{i=1}^N z_i \, \chi_i^N(x) \qquad\mbox{for } x\in\I,
\]
where $z\in\R^N$ and $\chi^N_i$ is the characteristic function of the interval $\I^N_i$.
By a slight abuse of notation, we also use the symbol $\QN$ for the operator
$\QN: \R^{N\times N} \to L^\infty(\II)$, given by
\[   
   \QN[B](x,y) := \sum_{i=1}^N\sum_{j=1}^N B_{ij} \, \chi^N_i(x) \, \chi^N_j(y) \qquad \mbox{for } (x,y)\in \I\times\I.
\]
For the adjacency matrix $W^N \in \R^{N\times N}$ of the graph $\GraphG^N$
and the matrix of edge lengths $\L^N \in \R^{N\times N}$
we denote
\[
   w^N:=\QN[W^N],
   \qquad
   \ell^N:=\QN[\L^N].
\]

We also use the cell-average projection $\ZN:L^1(\I)\to L^\infty(\I)$ given by
\begin{equation}   \label{def:ZN}
   \ZN[u](x):=N\int_{\I_i^N}u(s)\,\d s,
   \qquad x\in \I_i^N,
\end{equation}
and, again with a slight abuse of notation,
\[
   \ZN[u](x,y):=N^2\int_{\I_i^N}\int_{\I_j^N}u(s,t)\,\d s \, \d t,
   \qquad (x,y)\in \I_i^N\times \I_j^N.
\]
We note that, when restricted to $L^2(\I)$ and, respectively, to $L^2(\II)$, the operators $\ZN$ are orthogonal projections, and therefore self-adjoint.

\section{Assumptions and main result}\label{sec:main}

We shall work with a prescribed sequence of graphs
$\GraphG^N = (\Vset^N,\Eset^N)$ and source/sink terms $S^N\in \R^N_0$.
Let us recall that we denote $W^N\in \R^{N\times N}$ the adjacency matrix of $\GraphG^N$
and $w^N := \QN[W^N]$ its pixel picture.

We assume that there exists a function $\sigma\in \Ltwozero(\I)$ such that
the source/sink terms are given by 
\begin{equation} \label{ass:S}
   S_i^N := \int_{\I_i^N} \sigma(x) \, \d x \qquad \mbox{for } i\in[N].   
\end{equation}
Moreover, we assume that there exists $\lambda>0$ such that for all $N\in\N$,
    \begin{equation}  \label{ass:A1}
        \sum_{i=1}^N \sum_{j=1}^N  (z_i - z_j)^2 W^N_{ij} \geq \lambda N \sum_{i=1}^N z_i^2  \qquad\mbox{for all } z\in\R_0^N,
    \end{equation}
and that there exists $w\in L^\infty(\II)$ such that
    \begin{equation}  \label{ass:A2}
       w^N \to w \qquad \mbox{in } L^1(\II) \qquad \mbox{as } N\to\infty.
    \end{equation}
    
The spectral gap assumption \eqref{ass:A1} is equivalent to the lower bound $\mathfrak{f}(\GraphG^N) \geq \lambda N/2$
on the Fiedler number (algebraic connectivity) of $\GraphG^N$,
which is the second smallest eigenvalue of the matrix Laplacian of $W^N$; see, e.g., \cite{Fiedler, Mieghem}.
Therefore, \eqref{ass:A1} enforces uniform connectivity of the graphs $\GraphG^N$.

Assumption \eqref{ass:A2} restricts the validity of our results
to $0-1$ valued graphons (also called deterministic graphons), i.e., $\mbox{range}(w) \subset \{0, 1\}$.
This was also the setting treated in~\cite{alves2026rigorous}.
Then, the usual convergence of $w^N$ to $w$
in the cut norm indeed implies convergence in the $L^1$ norm topology;
see, e.g., \cite[Proposition 8.24]{Lovasz:2012}.
Of course, due to the uniform boundedness of $w^N$ in $L^\infty(\II)$,
the convergence \eqref{ass:A2} holds also in $L^q(\II)$ for any $1 < q<\infty$.

Moreover, \cite[Lemma 5.3]{alves2026rigorous} yields that \eqref{ass:A1} and \eqref{ass:A2} imply the spectral gap property for the limiting kernel $w=w(x,y)$,
\begin{equation}  \label{eq:SG}
    \iintII  (u(x) - u(y))^2 \, \d w(x, y) \geq \lambda \intI u(x)^2 \, \d x \qquad\mbox{for all } u\in L^2_0(\I).
\end{equation}

Finally, without loss of generality (after an eventual rescaling)
we assume that the edge lengths $\ell^N:=\QN[\L^N]$ are uniformly bounded,
\begin{equation}\label{ass:L1}
   \Norm{\ell^N}_{L^\infty(\II)}\le 1,
\end{equation}
and that there exists $\ell\in L^\infty(\II)$ such that
\begin{equation}\label{ass:L2}
   \ell^N\to \ell \qquad \mbox{in }L^1(\II) \qquad \mbox{as } N \to \infty.
\end{equation}

The main idea of our approach is to reformulate the kinetic energy functional $F^N_{\mathrm{kin}}[B]$,
defined in \eqref{eq:Fnkin}, as
\begin{equation}\label{def:Kdisc}
   G_{\mathrm{kin}}^N[B]:=
   \sup_{\Phi\in \R^N_0}
   \left\{
      2\sum_{i=1}^N S_i^N\Phi_i
      -\frac{1}{2N^2}\sum_{i=1}^N\sum_{j=1}^N W_{ij}^N B_{ij}(\Phi_j-\Phi_i)^2
   \right\},
\end{equation}
so that the total discrete energy reads
\begin{equation}\label{def:GNdisc}
G^N[B] := G_{\mathrm{kin}}^N[B] + F_{\mathrm{met}}^N[B],
\end{equation}
where $F_{\mathrm{met}}^N$ is the functional in~\eqref{eq:Fnmet}.
In Lemma \ref{lem:discrete-compatibility}  we establish the equality of the functionals~\eqref{def:Fn} and \eqref{def:GNdisc}, i.e., $G^N[B] = F^N[B]$ for all $B\in\BB^N_+$.
This includes the case $G^N[B] = F^N[B] = \infty$ when the Kirchhoff law \eqref{eq:K}
does not admit a solution  with the given matrix of conductivities $B\in\BB^N_+$.

In order to pass to the $\Gamma$-limit as $N\to\infty$, we reformulate the discrete functionals~\eqref{def:GNdisc}
in terms of semidiscrete integral functionals defined for $b\in L_+^1(\II)$ by
\begin{equation}\label{def:GN}
   \Gcont^N[b]:=\Gkin^N[b] + \Fmet^N[b],
\end{equation}
with
\begin{equation}\label{eq:GNkin}
   \Gkin^N[b]:=
   \sup_{\varphi\in \LinftyzeroN(\I)}
   \left\{
      2\intI \sigma(x)\varphi(x)\,\d x
      -\half \iintII b(x,y)(\dd \varphi(x,y))^2\,\d w^N(x,y)
   \right\}, 
   \end{equation}
and
\begin{equation} \label{eq:FNmet}
   \Fmet^N[b] :=\frac{1}{2\gamma}\iintII b(x,y)^\gamma\ell^N(x,y)^{\gamma+1}\,\d w^N(x,y).
\end{equation}
In Lemma \ref{lem:matrix-embedding}
we prove the equivalence of \eqref{def:GNdisc}
and \eqref{def:GN} in the sense that
\[
   \Gcont^N[\QN[B]]=G^N[B]
\]
for all $B\in \BB_+^N$ and $N\in\N$.

The limiting energy functional is defined for $b\in L_+^1(\II)$ by
\begin{equation}\label{def:Gcont}
   \Gcont[b]:=\Gkin[b]+\Fmet[b],
\end{equation}
with
\begin{equation}\label{eq:Gkin}
   \Gkin[b]:=
   \sup_{\varphi\in \Linftyzero(\I)}
   \left\{
      2\intI \sigma(x) \, \varphi(x)\,\d x
      -\half \iintII b(x,y)\big(\dd \varphi(x,y)\big)^2\,\d w(x,y)
   \right\}, 
\end{equation}
and
\begin{equation} \label{eq:Gmet}
   \Fmet[b]:=\frac{1}{2\gamma}\iintII b(x,y)^\gamma \, \ell(x,y)^{\gamma+1}\,\d w(x,y).
\end{equation}
We regard $\Gcont[b]$ as an extended-value functional on $L^1_+(\II)$.
Indeed, the metabolic term $\Fmet[b]$ may be infinite
if the nonnegative term $b^\gamma\ell^{\gamma+1} w$ is not integrable on $\II$.
Moreover, the kinetic part $\Gkin[b]$ may take the value $\infty$,
namely, if and only if the Poisson equation \eqref{eq:Poisson}
is not solvable with the kernel $b$.
We discuss this aspect in detail in Section~\ref{sec:Poisson}.

We now state the main result of this paper. We use Mosco convergence in its standard sense, namely the combination of a $\Gamma$-$\liminf$ condition along weakly convergent sequences and a $\Gamma$-$\limsup$ condition along strongly convergent recovery sequences; see~\cite{mosco1969convergence,mosco1994composite,Maso93}.

\begin{theorem}\label{thm:main}
Let $\gamma\geq 1$, $\sigma\in L^2_0(\I)$, and assume \eqref{ass:S}--\eqref{ass:L2}. Then the following statements hold.
\begin{enumerate}[label={\rm (\roman*)}]
\item For every sequence $\{b^N\}_{N\in\N}\subset L_+^1(\II)$ converging weakly in $L^1(\II)$ to $b\in L_+^1(\II)$, we have
\begin{equation}\label{eq:liminf-main}
   \Gcont[b]\le \liminf_{N\to\infty} \Gcont^N[b^N].
\end{equation}
\item For every $b\in L_+^\gamma(\II)$, there exists a sequence $\{b^N\}_{N\in\N}\subset L_+^\gamma(\II)$ converging strongly in $L^\gamma(\II)$ to $b$, and
\begin{equation}\label{eq:limsup-main}
   \Gcont[b]\ge \limsup_{N\to\infty} \Gcont^N[b^N].
\end{equation}
\end{enumerate}
Consequently, the sequence of energy functionals $\{\Gcont^N\}_{N\in\N}$ $\Gamma$-converges to $\Gcont$ in the sense of Mosco in $L^\gamma(\II)$.
\end{theorem}

We note that the $\Gamma$-$\liminf$ condition in Theorem~\ref{thm:main} is proved under the weaker assumption of weak convergence in $L^1(\II)$. Since $\II$ has finite measure, weak convergence in $L^\gamma(\II)$ implies weak convergence in $L^1(\II)$, and therefore this gives the $\Gamma$-$\liminf$ condition required for Mosco convergence in $L^\gamma(\II)$. On the other hand, the $\Gamma$-$\limsup$ condition is stated for $b\in L_+^\gamma(\II)$, which is the natural assumption needed to pass to the limit in the metabolic part of the energy.

\section{The Poisson equation and the dual formulation}\label{sec:Poisson}
We discuss the well-posedness and dual formulation of the Poisson equation \eqref{eq:Poisson},
which in its weak formulation with a test function  $\varphi$ reads
\begin{equation}  \label{eq:Poissonweak}
   \frac12 \iintII b(x,y)\dd p(x,y)\dd \varphi(x,y)\,\d w(x,y) = \intI \sigma(x) \,\varphi(x)\,\d x.
\end{equation}
Let us recall that $\sigma$ belongs to $\Ltwozero(\I)$ and that $w\in L^\infty(\II)$ satisfies the spectral gap condition~\eqref{eq:SG}.
We start by noting that for $b \in L^\infty_r(\II)$ with $r>0$,
the Lax--Milgram theorem immediately provides the existence of a unique
$p\in\Ltwozero(\I)$, so that \eqref{eq:Poissonweak} is verified for all $\varphi\in \Ltwozero(\I)$.
Moreover, we observe that \eqref{eq:Poissonweak} is the Euler--Lagrange equation
of the functional $\mathcal{J}_b : \Ltwozero \to \R$,
\begin{equation}   \label{def:Jb}
   \mathcal{J}_b[\varphi] := 2 \intI \sigma(x) \, \varphi(x) \,\d x - \frac{1}{2} \iintII b(x,y) \big(\dd\varphi(x,y)\big)^2 \d w(x,y),
\end{equation}
that is, whenever \eqref{eq:Poissonweak} admits a solution $p\in \Ltwozero$, we have
\[
   \sup_{\varphi\in \Ltwozero} \mathcal{J}_b[\varphi]  = \mathcal{J}_b[p] < \infty.
\]
In \cite[Lemma 5.5]{alves2026rigorous} the well-posedness of \eqref{eq:Poissonweak}
has been extended to $b \in L^2_r(\II)$ as follows.

\begin{lemma}\label{lem:Poisson-b}
Let $r>0$ and $b \in L^2_r(\II)$. Assume \eqref{ass:A1} and \eqref{ass:A2}. Then there exists a unique $p \in \Ltwozero(\I)$ such that~\eqref{eq:Poissonweak} holds
for all test functions $\varphi\in L^\infty_0(\I)$.
Moreover, 
\begin{equation}  \label{eq:testPoisson}
   \frac12 \iintII b \, (\dd p)^2 \, \d w(x,y) = \intI \sigma \,  p \, \d x < \infty,
\end{equation}
and
\begin{equation}   \label{est:p}
     \Norm{p}_{L^2(\I)}  \leq  \frac{2}{\lambda r} \Norm{\sigma}_{L^2(\I)}.
\end{equation}
\end{lemma}

We now study the case when $b=b(x,y)$ is only nonnegative on $\II$ (but strictly positive on the interior of $\II$),
but not uniformly bounded away from zero.
To gain an intuition, let us consider the example $b(x,y) := 2xy$, with $w(x,y) :=1$ for all $(x,y)\in\II$.
Then a straightforward calculation reveals the formal solution
\[
   p(x) = \frac{\sigma(x)}{x} - \intI \frac{\sigma(y)}{y} \d y,
\]
whenever the integral exists. Let us consider, for $\alpha\in (0,1)$, the choice
\begin{equation*}
  \sigma(x) := \begin{dcases}
  x^\alpha, \qquad & x\in (0,1/2), \\
    - \frac{2^{-\alpha}}{\alpha+1}, \qquad & x\in (1/2,1),
  \end{dcases}
\end{equation*}
which belongs to $\Linftyzero(\I)$. For this choice of $\sigma$, the integral
\[
   C:=\intI \frac{\sigma(y)}{y}\,\d y
\]
is finite. Moreover, for $x\in(0,1/2)$, the formal solution is given by
\[
   p(x)=x^{\alpha-1}-C.
\]
Now choose $q>1$ and set $\alpha:=\frac{q-1}{q}$, which gives $q(\alpha-1)=-1$.
For small enough $\delta\in(0,1/2)$ we have
\[
    |p(x)|\ge \frac12 x^{\alpha-1}  \qquad\mbox{for all }x\in(0,\delta).
\]
Consequently,
\[
   \int_0^1 |p(x)|^q\,\d x
   \ge
   2^{-q}\int_0^\delta x^{q(\alpha-1)}\,\d x
   =
   2^{-q}\int_0^\delta \frac{\d x}{x}
   =
   \infty.
\]
Therefore, no general $L^q$-regularity of the solution can be expected for any $q>1$
for general kernels $b=b(x,y)$.

Inspired by this observation, we define, for fixed $w\in L_+^\infty(\II)$ and $b\in L^1_+(\II)$, the bilinear form
\[
   Q(\varphi,\psi) := \frac12 \iintII b(x,y) \, \dd \varphi(x,y) \, \dd \psi(x,y)\,\d w(x,y) \qquad\mbox{for } \varphi, \psi \in L^\infty_0(\I).
\]
We denote $\mathcal{N}_Q := \big\{\varphi\in L^\infty_0(\I):\,  Q(\varphi,\varphi) = 0 \big\}$ the null space of the associated quadratic form.
Then $Q$ is an inner product on the quotient space $L^\infty_0(\I) / \mathcal{N}_Q$ and $\Norm{\varphi}_Q := Q(\varphi,\varphi)^{1/2}$ is the induced norm.
The natural space to look for solutions of \eqref{eq:Poissonweak} with $b\in L^1_+(\II)$ is then the (energy) Hilbert space
\begin{equation}   \label{def:HQ}
   \mathcal{H}_Q := \overline{ L^\infty_0(\I) / \mathcal{N}_Q }^{\Norm{\cdot}_Q},
\end{equation}
given by the completion of the quotient space $L^\infty_0(\I) / \mathcal{N}_Q$ in the norm $\Norm{\cdot}_Q$.
For the solvability of \eqref{eq:Poissonweak} we need that the linear functional
\begin{equation}   \label{eq:LinFun}
   \varphi \mapsto \intI \sigma(x) \, \varphi(x) \,\d x \qquad\mbox{for } \varphi\in \Linftyzero(\I)
\end{equation}
extends continuously to $\mathcal{H}_Q$. 
In other words, we need a constant $C_{Q,\sigma} >0$ such that
\begin{equation}  \label{cond:cont}
   \left| \intI \sigma(x) \, \varphi(x) \,\d x \right| \leq C_{Q,\sigma} \Norm{\varphi}_Q  \qquad\mbox{for all }  \varphi\in \Linftyzero(\I).
\end{equation}
This is essentially a compatibility condition between the (possibly degenerate) product kernel $b w\in L^1_+(\II)$
and the right-hand side $\sigma\in L^2_0(\I)$.
If it is met, then we denote by $\mathfrak{S}: \mathcal{H}_Q \to \R$ the continuous extension of \eqref{eq:LinFun}.
 The Riesz representation theorem immediately gives the following result.

\begin{lemma}\label{lem:genPoisson}
Let $w\in L_+^\infty(\II)$, $b\in L^1_+(\II)$ and $\sigma\in L^2_0(\I)$ be such that \eqref{cond:cont} holds.
Then there exists a unique $p\in \mathcal{H}_Q$ such that
\[   
   Q(p, \varphi) = \mathfrak{S}[\varphi] \qquad\mbox{for all }\varphi\in \mathcal{H}_Q.
\]
\end{lemma}

Moreover, introducing the extension of the functional \eqref{def:Jb},
\[
   \overline{\mathcal{J}_b}[\varphi] := 2 \mathfrak{S}[\varphi] - \Norm{\varphi}_Q^2 \qquad\mbox{for } \varphi\in \mathcal{H}_Q,
\]
the solution $p\in \mathcal{H}_Q$ constructed in Lemma \ref{lem:genPoisson} is its unique maximizer, i.e.,
\[
   \max_{\varphi\in\mathcal{H}_Q} \overline{\mathcal{J}_b}[\varphi] = \overline{\mathcal{J}_b}[p] = \Norm{p}_Q^2. 
\]
Consequently, if \eqref{cond:cont} holds, the definition \eqref{eq:Fcontkin} of $\Fkin$ can be meaningfully extended to $b\in L^1_+(\II)$
by setting $\Fkin[b] := \Norm{p}_Q^2$.
Since $\mathfrak{S}$ is a continuous extension of \eqref{eq:LinFun}, we have
\[
     \max_{\varphi\in\mathcal{H}_Q} \overline{\mathcal{J}_b}[\varphi] = \sup_{\varphi\in\Linftyzero(\I)} \mathcal{J}_b[\varphi],
\]
and we can finally write
\[
   \Fkin[b] = \Norm{p}_Q^2 = \sup_{\varphi\in\Linftyzero(\I)} \mathcal{J}_b[\varphi] = \Gkin[b],
\]
with $\mathcal{J}_b[\varphi]$ given by \eqref{def:Jb} and $\Gkin[b]$ defined in \eqref{eq:Gkin}.

If \eqref{cond:cont} does not hold, we have $\Fkin[b] := \Gkin[b] = \infty$.
Indeed, in this case either the functional
$\varphi\mapsto \int_\I \sigma\varphi\,\d x$ does not vanish on
$\mathcal N_Q$, in which case $\mathcal J_b$ is unbounded
along a null direction, or there exists a sequence $\varphi_n\in L^\infty_0(\I)$
with $\Norm{\varphi_n}_Q=1$ and $\left| \int_\I\sigma\varphi_n\, \d x \right| \to\infty$.

\section{Discrete and semidiscrete functionals}\label{sec:aux}

We first show the equivalence of the pressure-based definition~\eqref{def:Fn}
and the duality-based definition~\eqref{def:GNdisc} of the discrete energy functionals.

\begin{lemma}\label{lem:discrete-compatibility}
For every $N\in\N$ and all $B\in \BB_+^N$ we have
\begin{equation}\label{eq:Gdisc-equals-Fdisc}
   G^N[B] = F^N[B],
\end{equation}
with $F^N[B]$ defined in~\eqref{def:Fn} and $G^N[B]$ defined in~\eqref{def:GNdisc}.
\end{lemma}

\begin{proof}
It is sufficient to prove the equality for the kinetic part of the energy.
Let us observe that the Kirchhoff law \eqref{eq:K} is the Euler--Lagrange equation of the functional $J_B : \R^N_0 \to \R$,
\begin{equation}   \label{def:JB}
   J_B[\Phi] := \sum_{i=1}^N S_i^N \Phi_i - \frac{1}{4N^2}\sum_{i=1}^N\sum_{j=1}^N W_{ij}^N B_{ij}(\Phi_j-\Phi_i)^2,
\end{equation}
that is, whenever \eqref{eq:K} admits a solution $P\in\R^N_0$, we have
\[
   \sup_{\Phi\in\R^N_0} J_B[\Phi] = J_B[P].
\]
Multiplying~\eqref{eq:K} by $P_i$ and summing over $i\in[N]$ gives
\[
   \frac{1}{2N^2}\sum_{i=1}^N\sum_{j=1}^N W_{ij}^N B_{ij}(P_j-P_i)^2 = \sum_{i=1}^N S_i^N P_i.
\]
Consequently,
\[
   \sup_{\Phi\in\R^N_0} J_B[\Phi] = \frac{1}{4N^2}\sum_{i=1}^N\sum_{j=1}^N W_{ij}^N B_{ij}(P_j-P_i)^2,
\]
and for the kinetic part (first term) of the discrete energy functional \eqref{def:Fn} we have
\[
   \frac{1}{2N^2}\sum_{i=1}^N\sum_{j=1}^N W_{ij}^N B_{ij}(P_j-P_i)^2 = 2 \sup_{\Phi\in\R^N_0} J_B[\Phi] = G_{\mathrm{kin}}^N[B],
\]
with $G_{\mathrm{kin}}^N[B]$ defined in \eqref{def:Kdisc}.

On the other hand, if the Kirchhoff law \eqref{eq:K} does not admit a solution with the kernel $B\in \BB_+^N$,
then by the Fredholm alternative we have $G_{\mathrm{kin}}^N[B] = \infty$ and, by definition, $F_{\mathrm{kin}}^N[B] = \infty$.
\end{proof}

We now establish a connection between the discrete functional $G^N$ and
its semidiscrete counterpart $\Gcont^N$.

\begin{lemma}\label{lem:matrix-embedding}
For every $N\in\N$ and $B\in \BB_+^N$ we have
\begin{equation}\label{eq:matrix-embedding}
   \Gcont^N[\QN[B]]=G^N[B],
\end{equation}
with $\Gcont^N$ defined in \eqref{def:GN} and $G^N$ defined in \eqref{def:GNdisc}.
\end{lemma}

\begin{proof}
Let $b=\QN[B]$. Every $\varphi\in \LinftyzeroN(\I)$ can be written uniquely as $\varphi=\QN[\Phi]$ with $\Phi\in \R^N_0$.
Then, using \eqref{ass:S},
\[
   2\intI \sigma(x) \, \varphi(x)\,\d x
   =2\sum_{i=1}^N \Phi_i\int_{\I_i^N}\sigma(x)\,\d x
   =2\sum_{i=1}^N S_i^N\Phi_i.
\]
Moreover,
\[
   \half\iintII b(x,y)(\dd \varphi(x,y))^2\,\d w^N(x,y)
   =\frac{1}{2N^2}\sum_{i=1}^N\sum_{j=1}^N W_{ij}^N B_{ij}(\Phi_j-\Phi_i)^2.
\]
Taking the supremum over piecewise constant $\varphi\in \LinftyzeroN(\I)$ is therefore equivalent to taking the supremum over $\Phi\in \R^N_0$, and we obtain
\[
   \Gkin^N[\QN[B]]=G_{\mathrm{kin}}^N[B].
\]
The metabolic term $\Fmet^N[b]$ reads
\begin{align*}
   \frac{1}{2\gamma} \iintII & b(x,y)^\gamma \left(\ell^N(x,y) \right)^{\gamma+1} \d w^N(x, y)\\
    &=
     \frac{1}{2\gamma} \sum_{i=1}^N \sum_{j=1}^N W_{ij}^N B_{ij}^\gamma \left( \L_{ij}^N \right)^{\gamma+1} \int_{\I_i^N}\int_{\I_j^N} \d x \, \d y \\
       &= \frac{1}{2\gamma N^2} \sum_{i=1}^N \sum_{j=1}^N W_{ij}^N B_{ij}^\gamma \left( \L_{ij}^N \right)^{\gamma+1},
\end{align*}
where we used the fact that, by construction, $b= \QN[B] \equiv B_{ij}$ on the patch $\I_i\times \I_j$,
and similarly $\ell^N = \QN[\L^N] \equiv \L^N_{ij}$, $w^N = \QN[W^N] \equiv W^N_{ij}$ on $\I_i\times \I_j$.
Consequently, we have
\[
   \Fmet^N[\QN[B]]=F_{\mathrm{met}}^N[B],
\]
and \eqref{eq:matrix-embedding} follows.
\end{proof}

\section{Proof of the \texorpdfstring{$\Gamma$-liminf}{\unichar{"0393}-liminf} inequality}\label{sec:liminf}

Here we prove part (i) of Theorem~\ref{thm:main} under the assumptions stated therein.

\begin{lemma}\label{lem:kinetic-liminf}
Let $\{b^N\}_{N\in\N}\subset L_+^1(\II)$ converge weakly in $L^1(\II)$ to $b\in L_+^1(\II)$. Then
\begin{equation}\label{eq:kinetic-liminf}
   \Gkin[b]\le \liminf_{N\to\infty} \Gkin^N[b^N].
\end{equation}
\end{lemma}

\begin{proof}
Fix $\varphi\in \Linftyzero(\I)$ and set $\varphi^N:=\ZN[\varphi]$. By the definition \eqref{eq:GNkin} of $\Gkin^N$,
\begin{equation}\label{eq:test-discrete-kin}
   \Gkin^N[b^N]\ge 2\intI \sigma(x) \,\varphi^N(x)\,\d x-\half\iintII b^N(x,y) \left(\dd \varphi^N(x,y) \right)^2\,\d w^N(x,y).
\end{equation}
By the Lebesgue differentiation theorem, $\varphi^N\to \varphi$ almost everywhere on $\I$.
Since $\sigma\in L^2_0(\I)$ and
\[
   \sup_{N\in\N} \Norm{\varphi^N}_{L^\infty(\I)} =  \sup_{N\in\N} \Norm{\ZN[\varphi]}_{L^\infty(\I)} \leq \Norm{\varphi}_{L^\infty(\I)},
\]
we have by dominated convergence,
\begin{equation}\label{eq:source-conv}
   \intI \sigma(x) \, \varphi^N(x)\,\d x \to \intI \sigma(x) \, \varphi(x)\,\d x \qquad \text{as } N \to \infty.
\end{equation}
Moreover, Lemma \ref{lem:projection-test} of the Appendix gives the uniform boundedness of $(\dd \varphi^N)^2 w^N$ in $L^\infty(\II)$, and
\[  
   (\dd \varphi^N)^2 w^N \to (\dd \varphi)^2 w \qquad \mbox{in } L^q(\II) \qquad\mbox{as } N\to\infty,
\]
for any $1 \leq q<\infty$,
which implies convergence in measure.
An application of Lemma \ref{lem:gamma1} of the Appendix with $a^N := \left(\dd \varphi^N \right)^2 w^N$ and
$a:= \left(\dd \varphi \right)^2 w$ yields
\begin{equation}\label{eq:quadratic-conv}
   \iintII b^N(x,y) \left(\dd \varphi^N(x,y) \right)^2\,\d w^N(x,y)
   \to
   \iintII b(x,y) \left(\dd \varphi(x,y) \right)^2\,\d w(x,y).
\end{equation}
Passing to the limit inferior as $N\to\infty$ in \eqref{eq:test-discrete-kin}, using \eqref{eq:source-conv} and \eqref{eq:quadratic-conv}, we infer
\[
   \liminf_{N\to\infty}\Gkin^N[b^N]\ge 2\intI \sigma(x)\,\varphi(x)\,\d x-\half\iintII b(x,y) \left(\dd \varphi(x,y)\right)^2\,\d w(x,y).
\]
Since $\varphi\in \Linftyzero(\I)$ was arbitrary, taking the supremum over $\varphi$ yields \eqref{eq:kinetic-liminf}.
\end{proof}

For the metabolic part of the energy we have the following result.

\begin{lemma}\label{lem:metabolic-liminf}
Let $\{b^N\}_{N\in\N}\subset L_+^1(\II)$ converge weakly in $L^1(\II)$ to $b\in L_+^1(\II)$. Then
\begin{equation}\label{eq:metabolic-liminf}
   \Fmet[b]\le \liminf_{N\to\infty}\Fmet^N[b^N].
\end{equation}
\end{lemma}

\begin{proof}
We first prove that
\begin{equation}   \label{eq:first}
   b^N \left(\ell^N\right)^\frac{\gamma+1}{\gamma} w^N \;\rightharpoonup\; b \, \ell^\frac{\gamma+1}{\gamma} w
   \qquad
   \mbox{weakly in } L^1(\II).
\end{equation}
Indeed, fixing any test function $\varphi\in L^\infty(\II)$ and denoting
\[
   a^N := \left(\ell^N\right)^\frac{\gamma+1}{\gamma} w^N \varphi,
   \qquad
   a := \ell^\frac{\gamma+1}{\gamma} w \, \varphi,
\]
we have that $\{a^N\}_{N \in \N}$ is uniformly bounded in $L^\infty(\II)$ and converges to $a\in L^\infty(\II)$ in $L^q(\II)$ with any $1 \leq q<\infty$.
This implies $a^N\to a$ in measure, and Lemma \ref{lem:gamma1} of the Appendix gives
\[
   \iintII b^N a^N\,\d x \, \d y \to \iintII b \, a\,\d x \, \d y \qquad\mbox{as } N\to\infty.
\]
As the choice of the test function $\varphi\in L^\infty(\II)$ is arbitrary, we obtain \eqref{eq:first}.

Then, with the definitions \eqref{eq:FNmet} and \eqref{eq:Gmet}, and since $w^\gamma=w$, $(w^N)^\gamma=w^N$, we deduce
\begin{align*}
 \Fmet[b] &= \frac{1}{2\gamma} \Norm{b \, \ell^\frac{\gamma+1}{\gamma} w}^\gamma_{L^\gamma(\II)}   \\
   &\leq \liminf_{N\to\infty} \frac{1}{2\gamma} \Norm{b^N (\ell^N)^\frac{\gamma+1}{\gamma} w^N}^\gamma_{L^\gamma(\II)} \\
   &= \liminf_{N\to\infty} \Fmet^N[b^N],
\end{align*}
where we used Lemma \ref{lem:weakLSC} of the Appendix.
\end{proof}

The claim of Theorem~\ref{thm:main}(i) follows immediately by
combining the results of Lemma~\ref{lem:kinetic-liminf} and Lemma~\ref{lem:metabolic-liminf},
\begin{align*}
\Gcont[b]& =\Gkin[b]+\Fmet[b] \\
  & \le \liminf_{N\to\infty}\Gkin^N[b^N]+\liminf_{N\to\infty}\Fmet^N[b^N] \\
   & \le \liminf_{N\to\infty}\Gcont^N[b^N].
\end{align*}

\section{Proof of the \texorpdfstring{$\Gamma$-limsup}{\unichar{"0393}-limsup} inequality}\label{sec:limsup}

In this section we prove Theorem~\ref{thm:main}(ii).
The argument is based on a uniformly positive regularization of the conductivity kernels
and on the recovery estimate along the canonical approximations $b^N=\ZN[b]$
on the uniformly positive class.

\begin{lemma}\label{lem:GNkin}
Fix $r>0$ and let $b\in L^1_r(\II)$.
Then the sequence $b^N = \ZN[b]$ converges in $L^1(\II)$ to $b$ as $N\to\infty$, and
\[
   \limsup_{N\to\infty}\Gkin^N[b^N] \le \Gkin[b].
\]
\end{lemma}

\begin{proof}
The strong convergence $b^N\to b$ in $L^1(\II)$ is a standard result in approximation theory.
Moreover, we obviously have $b^N\in L^1_r(\II)$ for all $N\in\N$.

We define, for $\varphi\in \LinftyzeroN(\I)$,
\[
   \mathcal{J}^N[\varphi]
   :=
   2\intI \sigma(x) \, \varphi(x)\,\d x
   -\frac12\iintII b^N(x,y)(\dd\varphi(x,y))^2\,\d w^N(x,y).
\]
Therefore,
\[
   \Gkin^N[b^N]=\sup_{\varphi\in\LinftyzeroN(\I)} \mathcal{J}^N[\varphi],
   \qquad
   \Gkin[b]=\sup_{\psi\in\Linftyzero(\I)} \mathcal{J}_{b}[\psi],
\]
with $\mathcal{J}_{b}[\psi]$ defined in \eqref{def:Jb}.
The discrete spectral gap assumption \eqref{ass:A1} gives
\[
   \iintII (\dd \varphi(x,y))^2\,\d w^N(x,y)
   \ge
   \lambda \Norm{\varphi}_{L^2(\I)}^2
   \qquad
   \mbox{for all } \varphi\in\LinftyzeroN(\I),
\]
and using $b^N\ge r$, we have
\[
   \iintII b^N(x,y) (\dd \varphi(x,y))^2\,\d w^N(x,y)
   \ge
   r\lambda \Norm{\varphi}_{L^2(\I)}^2
   \qquad
   \mbox{for all } \varphi\in\LinftyzeroN(\I).
\]
Consequently, $\mathcal{J}^N$ is a strictly concave coercive functional
on the finite-dimensional space $\LinftyzeroN(\I)$, and the supremum in the
definition of $\Gkin^N[b^N]$ is attained. Let $p^N\in\LinftyzeroN(\I)$ be the
maximizer,
\[
   \Gkin^N[b^N]= \mathcal{J}^N[p^N].
\]
Then we have
\[
   0 = \mathcal{J}^N[0]
   \le
   \mathcal{J}^N[p^N]
   \le
   2\Norm{\sigma}_{L^2(\I)} \Norm{p^N}_{L^2(\I)}
   -\frac{r\lambda}{2}\Norm{p^N}_{L^2(\I)}^2,
\]
which gives the uniform bound
\[
   \Norm{p^N}_{L^2(\I)} \leq \frac{4}{r\lambda} \Norm{\sigma}_{L^2(\I)}.
\]
After an eventual extraction of a subsequence, we have
$p^N \rightharpoonup p$ weakly in $L^2(\I)$
for some $p\in L^2_0(\I)$.
We claim that
\begin{equation}\label{eq:weighted-lsc}
   \iintII b(x,y)(\dd p(x,y))^2\,\d w(x,y)
   \le
   \liminf_{N\to\infty}
   \iintII b^N(x,y)(\dd p^N(x,y))^2\,\d w^N(x,y).
\end{equation}
Indeed, fix $M>0$ and define
\[
   a_M^N(x,y):=\min\big\{b^N(x,y),M\big\}\,w^N(x,y),
   \qquad
   a_M(x,y):=\min\big\{b(x,y),M\big\}\,w(x,y).
\]
Since $b^N\to b$ in $L^1(\II)$ and $\min\big\{b^N,M\big\}$ is uniformly bounded in $L^\infty(\II)$, we have
\[
   \min\{b^N,M\} \to \min\{b,M\}
   \qquad\mbox{in }L^q(\II) \qquad 
   \mbox{for every } 1 \leq q<\infty.
\]
With $w^N\to w$ in $L^q(\II)$ for every $1 \leq q<\infty$, we readily have
\[
   a_M^N\to a_M
   \qquad\mbox{in }L^q(\II) \qquad 
   \mbox{for every } 1 \leq q<\infty.
\]
Then the weak convergence of $\dd p^N$ to $\dd p$ in $L^2(\II)$ 
combined with the uniform boundedness of $a_M^N$ in $L^\infty(\II)$
implies that
\[
   \sqrt{a_M^N} \, \dd p^N
   \rightharpoonup
   \sqrt{a_M} \, \dd p
   \qquad\mbox{weakly in }L^2(\II).
\]
Hence, by weak lower semicontinuity of the $L^2$-norm,
\[
   \iintII a_M(x,y)(\dd p(x,y))^2\,\d x\d y
   \le
   \liminf_{N\to\infty}
   \iintII a_M^N(x,y)(\dd p^N(x,y))^2\,\d x\d y.
\]
Since, by definition, $a_M^N\le b^Nw^N$, we get
\begin{align*}
   \iintII & \min\{b(x,y),M\}(\dd p(x,y))^2\,\d w(x,y) \\
   & \le
   \liminf_{N\to\infty}
   \iintII b^N(x,y) (\dd p^N(x,y))^2\,\d w^N(x,y).
\end{align*}
Letting $M\to\infty$ and using the monotone convergence theorem yields
\eqref{eq:weighted-lsc}.

Finally, since
$p^N\rightharpoonup p$ weakly in $L^2(\I)$ and $\sigma\in L^2_0(\I)$,
\[
   \intI \sigma(x) \, p^N(x)\,\d x \to \intI \sigma(x) \, p(x)\,\d x,
\]
and using \eqref{eq:weighted-lsc}, we obtain
\begin{align*}
  \limsup_{N\to\infty} & \ \Gkin^N[b^N] \\
   &=
   \limsup_{N\to\infty}
   \left(
      2\intI \sigma(x) p^N(x) \,\d x
      -\frac12\iintII b^N(x,y) (\dd p^N(x,y))^2\,\d w^N(x,y)
   \right)
   \\
   &\le
   2\intI \sigma(x) \, p(x)\,\d x
   -\frac12\iintII b(x,y)\big(\dd p(x,y)\big)^2\,\d w(x,y) \\
   & \leq \Gkin[b],
\end{align*}
where the last inequality follows from the definition \eqref{eq:Gkin} of $\Gkin[b]$,
approximating $p\in L^2_0(\I)$ by a sequence of bounded zero-mean truncations.
\end{proof}

\begin{lemma}\label{lem:GNmet}
Let $\gamma\ge1$. Fix $r>0$ and let $b\in L^\gamma_r(\II)$.
Then the sequence $b^N = \ZN[b]$ converges to $b$ in $L^\gamma(\II)$ as $N\to\infty$, and
\[
   \lim_{N\to\infty}\Fmet^N[b^N] = \Fmet[b].
\]
\end{lemma}

\begin{proof}
The strong convergence $b^N\to b$ in $L^\gamma(\II)$ is a standard approximation result. 

According to the definition \eqref{eq:FNmet}, we have
\begin{align*}
 \lim_{N\to\infty} \Fmet^N[b^N]
   &= \lim_{N\to\infty}  \frac{1}{2\gamma} \iintII \big( b^N(x,y) \big)^\gamma \big( \ell^N(x,y) \big)^{\gamma+1} \d w^N(x, y)  \\
   &= \frac{1}{2\gamma}  \iintII \big( b(x,y) \big)^\gamma \big( \ell(x,y) \big)^{\gamma+1} \d w(x, y)  \\
   &= \Fmet[b],
\end{align*}
using Lemma~\ref{lem:strongConv} with $(b^N)^\gamma\to b^\gamma$ in $L^1(\II)$, with
\[
   (\ell^N)^{\gamma+1}w^N \to \ell^{\gamma+1}w
   \qquad\text{in }L^q(\II)
   \qquad\text{for every }1\le q<\infty,
\]
and with the uniform boundedness of $(\ell^N)^{\gamma+1}w^N$ in $L^\infty(\II)$.
\end{proof}

The following lemma establishes the continuity of $\Gcont$ with respect to the positive regularization of the kernel.

\begin{lemma}\label{lem:delta-continuity}
Let $b\in L_+^1(\II)$ be such that $\Gcont[b]< \infty$, and for $r>0$ define $b^r:=b+r$.
Then
\[
   \Gcont[b^r]\to \Gcont[b]
   \qquad \text{as } r\to 0^+.
\]
\end{lemma}

\begin{proof}
Since $b^r\ge b$ on $\II$, we have for any $\varphi\in \Linftyzero(\I)$,
\[
   \mathcal{J}_{b^r}[\varphi] \le \mathcal{J}_b[\varphi],
\]
with $\mathcal{J}_b[\varphi]$ defined in \eqref{def:Jb}.
Taking the supremum over $\varphi\in \Linftyzero(\I)$ gives
\[
   \Gkin[b^r]\le \Gkin[b].
\]
Fix $\eps>0$ and choose $\varphi_\eps\in \Linftyzero(\I)$ such that
\[
   \Gkin[b]\le  \mathcal{J}_b[\varphi_\eps] +\eps.
\]
Then
\[
   \Gkin[b^r]\ge \mathcal{J}_{b^r}[\varphi_\eps]
   = \mathcal{J}_b[\varphi_\eps] - \frac{r}{2}\iintII \big(\dd \varphi_\eps(x,y)\big)^2\,\d w(x,y),
\]
hence
\[
   \Gkin[b^r]\ge \Gkin[b]-\eps-\frac{r}{2}\iintII \big(\dd \varphi_\eps\big)^2\,\d w.
\]
Letting $r\to 0^+$ and then $\eps\to 0^+$ yields the convergence of $\Gkin[b^r]$ to $\Gkin[b]$.

Moreover, $b^r\to b$ a.e.\ on $\II$ as $r\to0^+$, and for $r<1$ we have
\[
   (b^r)^\gamma \ell^{\gamma+1}w
   \le 2^{\gamma-1}(1+b^\gamma)\ell^{\gamma+1}w.
\]
The right-hand side is integrable, since $\Fmet[b]\le \Gcont[b]<+\infty$ by assumption and $\ell^{\gamma+1}w\in L^\infty(\II)$. Therefore, the dominated convergence theorem gives
\[
   \Fmet[b^r]\to \Fmet[b].
\]
Hence $\Gcont[b^r]\to \Gcont[b]$.
\end{proof}

\medskip

\begin{proof}[Proof of Theorem~\ref{thm:main}(ii)]
If $\Gcont[b]=\infty$, there is nothing to prove; for instance, we may take $b^N:=b$ for every $N\in\N$. We therefore assume that
$\Gcont[b]<\infty$.

For $k\in\N$, set
\[
   b^{(k)}:=b+\frac1k \in L^\gamma_{1/k}(\II).
\]
By Lemma~\ref{lem:delta-continuity},
\begin{equation}\label{eq:delta-cont-main}
   \Gcont[b^{(k)}]\to \Gcont[b]
   \qquad \text{as } k\to\infty .
\end{equation}
Moreover, for each fixed $k\in\N$, Lemmas \ref{lem:GNkin} and \ref{lem:GNmet} give
\[  
   \limsup_{N\to\infty}
   \Gcont^N[\ZN[b^{(k)}]]
   \le
   \Gcont[b^{(k)}].
\]
Hence, for each $k\in\N$, there exists $N_k\in\N$ such that for all
$N\ge N_k$,
\begin{equation}\label{eq:block-bound-1}
   \Gcont^N[\ZN[b^{(k)}]]
   \le
   \Gcont[b^{(k)}]+\frac1k .
\end{equation}
Passing to larger values if necessary, we may assume that the sequence
$\{N_k\}_{k\in\N}$ is strictly increasing.

For the finitely many indices $N<N_1$, choose $b^N\in L^\gamma_+(\II)$ arbitrarily. For $N\ge N_1$, let $k(N)\in\N$ be the unique index such that
\[
   N_{k(N)}\le N< N_{k(N)+1},
\]
and define
\begin{equation}\label{eq:diagonal-sequence}
   b^N:=\ZN[b^{(k(N))}] \in L^\gamma_+(\II).
\end{equation}
Then $b^N\to b$ in $L^\gamma(\II)$ as $N\to\infty$. Indeed,
\[
   \Norm{b^N-b}_{L^\gamma(\II)}
   \le
   \Norm{\ZN[b^{(k(N))}-b]}_{L^\gamma(\II)}
   +
   \Norm{\ZN[b]-b}_{L^\gamma(\II)}.
\]
While the second term vanishes in the limit $N\to\infty$ by the standard approximation properties of $\ZN$, for the first term we have
\[
   \Norm{\ZN[b^{(k(N))} - b]}_{L^\gamma(\II)} \leq \Norm{b^{(k(N))} - b}_{L^\gamma(\II)} = \frac{1}{k(N)} \to 0,
\]
as $N\to\infty$.

Now, by \eqref{eq:block-bound-1}, for every $N\ge N_1$,
\[
   \Gcont^N[b^N]
   =
   \Gcont^N[\ZN[b^{(k(N))}]]
   \le
   \Gcont[b^{(k(N))}]+\frac1{k(N)}.
\]
Therefore
\[
   \limsup_{N\to\infty}\Gcont^N[b^N]
   \le
   \limsup_{N\to\infty}\left(\Gcont[b^{(k(N))}]+\frac{1}{k(N)}\right).
\]
Since $k(N)\to\infty$ as $N\to\infty$, \eqref{eq:delta-cont-main} gives
\[
   \Gcont[b^{(k(N))}]\to \Gcont[b],
   \qquad
   \frac{1}{k(N)}\to0 \qquad \text{as } N \to \infty.
\]
Hence
\[
   \limsup_{N\to\infty}\Gcont^N[b^N]\le \Gcont[b],
\]
which completes the proof.
\end{proof}

\appendix
\section{Auxiliary convergence results}\label{appendix}

\begin{lemma}\label{lem:projection-test}
Let $v\in \Linftyzero(\I)$ and set $v^N:=\ZN[v]$. Then $v^N\in \LinftyzeroN(\I)$, $\Norm{v^N}_{L^\infty(\I)}\le \Norm{v}_{L^\infty(\I)}$, and
\begin{equation*}
   v^N\to v \qquad \text{in }L^q(\I) \qquad \text{for every } 1\le q<\infty.
\end{equation*}
Moreover,
\begin{equation*}
   (\dd v^N)^2 w^N \to (\dd v)^2 w
   \qquad \text{in }L^{q}(\II) \qquad \text{for every } 1\le q<\infty.
\end{equation*}
\end{lemma}

\begin{proof}
By construction \eqref{def:ZN}, $v^N=\ZN[v]$ is piecewise constant on each interval $\I_i^N$,
and has zero average over $\I$.
Therefore, $v^N\in \LinftyzeroN(\I)$ and
\[
   \Norm{v^N}_{L^\infty(\I)}\le \Norm{v}_{L^\infty(\I)} \qquad\mbox{for all } N\in\N.
\]
The Lebesgue differentiation theorem yields
$v^N(x)\to v(x)$ for almost all $x\in \I$,
and, consequently, the dominated convergence theorem gives
\[
v^N\to v
\qquad\text{in }L^q(\I)
\qquad\text{for every }1\le q<\infty.
\]

Next, for any $(x,y)\in\II$,
\[
   |\dd v^N(x,y)-\dd v(x,y)|   \le |v^N(y)-v(y)|+|v^N(x)-v(x)|.
\]
Hence, for every $q\ge1$,
\[
|\dd v^N(x,y)-\dd v(x,y)|^q
\le
2^{q-1}\Bigl(|v^N(y)-v(y)|^q+|v^N(x)-v(x)|^q\Bigr).
\]
Integrating over $\II$, we obtain
\[
\|\dd v^N-\dd v\|_{L^q(\II)}^q
\le
2^q\|v^N-v\|_{L^q(\I)}^q,
\]
and therefore
\[
\dd v^N\to \dd v
\qquad\text{in }L^q(\II)
\qquad\text{for every }1\le q<\infty.
\]
Due to the uniform bound
\[
   \Norm{\dd v^N}_{L^\infty(\II)}  \leq   2\Norm{v^N}_{L^\infty(\I)}  \leq  2\Norm{v}_{L^\infty(\I)},
\]
it follows that
\[
(\dd v^N)^2\to (\dd v)^2
\qquad\text{in }L^q(\II)
\qquad\text{for every }1\le q<\infty.
\]

By \eqref{ass:A2} and the bound $\Norm{w^N}_{L^\infty(\II)}\le 1$, we have
\[
w^N\to w
\qquad\text{in }L^q(\II)
\qquad\text{for every }1\le q<\infty.
\]
Therefore, with the uniform $L^\infty$-bounds on $(\dd v^N)^2$ and $w$, we obtain
\begin{align*}
\big\|(\dd v^N)^2w^N-(\dd v)^2w\big\|_{L^q(\II)}
\le \ &
\big\|(\dd v^N)^2\big\|_{L^\infty(\II)}\|w^N-w\|_{L^q(\II)} \\
& +
\|w\|_{L^\infty(\II)}\big\|(\dd v^N)^2-(\dd v)^2\big\|_{L^q(\II)},
\end{align*}
and the claim follows.
\end{proof}

\begin{lemma}\label{lem:gamma1}
Let $b^N\rightharpoonup b$ weakly in $L^1(\II)$ and
$\{a^N\}_{N\in\N}\subset L^\infty(\II)$ be such that
\[
    \sup_{N\in\N} \|a^N\|_{L^\infty(\II)}<\infty, \qquad
   a^N\to a \quad\text{in measure on } \II,
\]
for some $a\in L^\infty(\II)$.
Then
\[
   \iintII b^N a^N\,\d x \, \d y \to \iintII b \, a\,\d x \, \d y \qquad\mbox{as } N\to\infty.
\]
\end{lemma}

\begin{proof}
We write
\[
   \iintII (b^N a^N - ba) \,\d x \, \d y 
   =
   \iintII b^N(a^N-a) \,\d x \, \d y 
   +
   \iintII (b^N-b)a \,\d x \, \d y.
\]
The second term tends to zero as $N\to\infty$ due to the weak convergence
$b^N\rightharpoonup b$ in $L^1(\II)$. Therefore we need to prove that
\begin{equation*}  
   \iintII b^N(a^N-a)\,\d x \, \d y \to 0 \qquad \text{as }N \to \infty.
\end{equation*}
Since $b^N\rightharpoonup b$ weakly in $L^1(\II)$, the sequence
$\{b^N\}_{N\in\N}$ is bounded in $L^1(\II)$ and we denote
\[
   C:=\sup_{N\in\N} \Norm{b^N}_{L^1(\II)} < \infty.
\]
Moreover, by the Dunford--Pettis theorem, $\{b^N \}_{N \in \N}$ is uniformly integrable, that is, for every $\varepsilon>0$ there exists $\delta>0$ such that,
for any measurable set $E\subset\II$ with $|E|<\delta$ one has
\[
   \sup_{N\in\N} \iint_E |b^N|\,\d x \, \d y <\varepsilon.
\]
Since $a^N\to a$ in measure, the sets
\[
   E_N^\varepsilon:=\big\{(x,y)\in \II:\, |a^N(x,y)-a(x,y)|>\varepsilon\big\}, \qquad \varepsilon > 0,
\]
satisfy $|E_N^\varepsilon|\to 0$ as $N \to \infty$. Hence for $N\in\N$ sufficiently large,
$|E_N^\varepsilon|<\delta$. Then
\[
\begin{aligned}
   \left|\iintII b^N(a^N-a)\,\d x \,\d y \right|
   & \le
   \iint_{\II\setminus E_N^\varepsilon}|b^N|\,|a^N-a|\,\d x \, \d y
   +
   \iint_{E_N^\varepsilon}|b^N|\,|a^N-a|\,\d x \, \d y  \\
   &\le
   \varepsilon \|b^N\|_{L^1(\II)}
   +
   M\iint_{E_N^\varepsilon}|b^N|\,\d x \, \d y  \\
   &\le
   (C+M) \varepsilon,
\end{aligned}
\]
where we denoted
\[
   M:=\sup_{N\in\N} \|a^N\|_{L^\infty(\II)}+\|a\|_{L^\infty(\II)}<\infty .
\]
Since $\varepsilon>0$ was arbitrary, the result follows.
\end{proof}

\begin{lemma}\label{lem:weakLSC}
Let $a^N \rightharpoonup a$ weakly in $L^1(\II)$ as $N\to\infty$.
Then for any $\gamma\geq 1$ we have
\[
   \Norm{a}_{L^\gamma(\II)} \leq \liminf_{N\to\infty} \Norm{a^N}_{L^\gamma(\II)}.
\]
\end{lemma}

\begin{proof}
Denote
\[
   L := \liminf_{N\to\infty} \Norm{a^N}_{L^\gamma(\II)}.
\]
If $L=\infty$, then there is nothing to prove; therefore we may assume $L < \infty$.

We first consider the case $\gamma=1$. Since $a^N\rightharpoonup a$ weakly in $L^1(\II)$, for every $\varphi\in L^\infty(\II)$ with
$\Norm{\varphi}_{L^\infty(\II)}\le 1$ we have
\[
   \left|\iintII a \, \varphi\,\d x\,\d y\right|
   =
   \left|\lim_{N\to\infty}\iintII a^N \varphi\,\d x\,\d y\right|
   \le
   \liminf_{N\to\infty}\Norm{a^N}_{L^1(\II)}.
\]
Taking the supremum over all such $\varphi$ and using the dual representation of the $L^1$ norm gives
\[
   \Norm{a}_{L^1(\II)}
   \le
   \liminf_{N\to\infty}\Norm{a^N}_{L^1(\II)}.
\]

Let now $\gamma>1$ and set $\gamma':=\frac{\gamma}{\gamma-1}$. For any $\varphi\in L^\infty(\II)$, weak convergence in $L^1(\II)$ gives
\[
   \iintII a \, \varphi\,\d x\,\d y
   =
   \lim_{N\to\infty}\iintII a^N\varphi\,\d x\,\d y.
\]
Hence, by H\"older's inequality,
\begin{align*}
   \left| \iintII a \, \varphi \,\d x \, \d y  \right|
   &\leq
     \liminf_{N\to\infty} \Norm{a^N}_{L^\gamma(\II)}
     \Norm{\varphi}_{L^{\gamma'}(\II)} \\
   &= L \Norm{\varphi}_{L^{\gamma'}(\II)}.
\end{align*}
Since $L^\infty(\II)$ is dense in $L^{\gamma'}(\II)$, the functional
\[
   \varphi\mapsto \iintII a \, \varphi\,\d x\,\d y
\]
extends continuously to $L^{\gamma'}(\II)$ with operator norm at most $L$.
By the duality between $L^\gamma(\II)$ and $L^{\gamma'}(\II)$, we conclude that $a\in L^\gamma(\II)$ and
\[
   \Norm{a}_{L^\gamma(\II)}
   =
   \sup_{\Norm{\varphi}_{L^{\gamma'}(\II)}\leq 1}
   \left| \iintII a \, \varphi\,\d x\,\d y \right|
   \le L.
\]
This proves the claim.
\end{proof}

\begin{lemma}\label{lem:strongConv}
Let the sequence $\{a^N\}_{N\in\N}$ be uniformly bounded in $L^\infty(\II)$
and converge in $L^1(\II)$ to $a\in L^\infty(\II)$. Moreover, let $b^N \to b$ in $L^1(\II)$ as $N\to\infty$.
Then
\[
   \iintII a^N b^N \, \d x \, \d y \to \iintII a \, b \, \d x \, \d y \qquad\mbox{as } N\to\infty.
\]
\end{lemma}

\begin{proof}
We write
\[
   \iintII \left( a^N b^N -  ab \right) \d x \, \d y  =
      \iintII a^N(b^N-b)\,\d x \, \d y  + \iintII b(a^N-a)\,\d x \, \d y.
\]
The first term of the right-hand side vanishes in the limit due to the strong convergence of $b^N$ to $b$ in $L^1(\II)$
and the uniform boundedness of $a^N$ in $L^\infty(\II)$.

Denote
\[
   C:= \|a\|_{L^\infty(\II)} + \sup_{N\in\mathbb N}\|a^N\|_{L^\infty(\II)}<\infty .
\]
Fix $K>0$. Then
\[
\begin{aligned}
\left|\iintII b(a^N-a)\,\d x \, \d y\right|
&\le
\iint_{\{|b|\le K\}} |b|\,|a^N-a|\,\d x \, \d y
+
\iint_{\{|b|>K\}} |b|\,|a^N-a|\,\d x \, \d y  \\
&\le
K\|a^N-a\|_{L^1(\II)}
+
C\iint_{\{|b|>K\}} |b|\,\d x \, \d y .
\end{aligned}
\]
Taking $N\to\infty$ gives
\[
   \limsup_{N\to\infty}  \left|\iintII b(a^N-a)\,\d x \, \d y\right|
   \le C\iint_{\{|b|>K\}} |b|\,\d x \, \d y,
\]
and the right-hand side vanishes as $K\to\infty$ due to $b\in L^1(\II)$.
\end{proof}

{\small
\section*{Acknowledgments}
This publication is based upon work partially supported by KAUST under Award No.\ ORFS-CRG12-2024-6430.}

\end{document}